\documentclass[11pt]{amsart}

\usepackage{graphicx, amssymb, pinlabel}
\usepackage[english]{babel}
\newtheorem{theorem}{Theorem}[section]
\newtheorem{proposition}[theorem]{Proposition}
\newtheorem{lemma}[theorem]{Lemma}
\newtheorem{corollary}[theorem]{Corollary}
\theoremstyle{definition}
\newtheorem{definition}[theorem]{Definition}
\newtheorem{notation}[theorem]{Notation}
\newtheorem{example}[theorem]{Example}
\newtheorem{remark}[theorem]{Remark}

\numberwithin{equation}{section}

\newcommand{\Homeo}{\operatorname{Homeo}}
\newcommand{\Mod}{\operatorname{Mod}}
\newcommand{\lcm}{\operatorname{lcm}}
\renewcommand{\O}{\operatorname{{\mathcal O}}}

\begin{document}

\title[Roots of Dehn twists about separating curves]
{Roots of Dehn twists\\ about separating curves}

\author{Kashyap Rajeevsarathy}
\address{Department of Mathematics\\
University of Oklahoma\\
Norman, Oklahoma 73019\\
USA}
\urladdr{www.math.ou.edu/$_{\widetilde{\phantom{n}}}$kashyap/}
\email{kashyap@math.ou.edu}
\date{\today}
\keywords{surface, mapping class, Dehn twist, separating curve, root}

\begin{abstract}
Let $C$ be a curve in a closed orientable surface $F$ of genus $g \geq 2$ that separates $F$ into subsurfaces $\widetilde {F_i}$ of genera $g_i$, for $i = 1,2$. We study the set of roots in $\Mod(F)$ of the Dehn twist $t_C$ about $C$. All roots arise from pairs of $C_{n_i}$-actions on the $\widetilde{F_i}$, where $n=\lcm(n_1,n_2)$ is the degree of the root, that satisfy a certain compatibility condition. The $C_{n_i}$ actions are of a kind that we call nestled actions, and we classify them using tuples that we call data sets. The compatibility condition can be expressed by a simple formula, allowing a classification of all roots of $t_C$ by compatible pairs of data sets. We use these data set pairs to classify all roots for $g = 2$ and $g = 3$. We show that there is always a root of degree at least $2g^2+2g$, while $n \leq 4g^2+2g$. We also give some additional applications.
\end{abstract}

\maketitle

\section{Introduction}

Let $F$ be a closed orientable surface of genus $g \geq 2$ and $C$ be a simple closed curve in $F$. Let $t_C$ denote a left handed Dehn twist about $C$.

When $C$ is a nonseparating curve, the existence of roots of $t_C$ is not so apparent. In their paper~\cite{MS}, D. Margalit and S. Schleimer showed the existence of such roots by finding elegant examples of roots of $t_C$ whose degree is $2g + 1$ on a surface of genus $g+1$. This motivated an earlier collaborative work with D. McCullough~\cite{MK1} in which we derived necessary and sufficient conditions for the existence of a root of degree $n$. As immediate applications of the main theorem in the paper, we showed that roots of even degree cannot exist and that $n \leq 2g + 1$. The latter shows that the Margalit-Schleimer roots achieve the maximum value of $n$ among all the roots for a given genus.

Suppose that $C$ is a  curve that separates $F$ into subsurfaces $\widetilde{F}_i$ of genera $g_i$ for $i = 1, 2$. It is evident that roots of $t_C$ exist. As a simple example, for the closed orientable surface of genus 2, we can obtain a square root of the Dehn twist $t_C$ by rotating one of the subsurfaces on either side of $C$ by an angle $\pi$, producing a half-twist near $C$. As in the case for nonseparating curves, a natural question is whether we can give necessary and sufficient conditions for the existence of a degree $n$ root of $t_C$. In this paper, we derive such conditions and apply them to obtain information about the possible degrees. We use Thurston's orbifold theory~\cite[Chapter 13]{T1} to prove the main result. A good reference for this theory is P.Scott~\cite{S1}.

We start by defining a special class of $C_n$-actions called \textit{nestled $(n, \ell)$-actions}. These $C_n$-actions have a distinguished fixed point and the points fixed by some nontrivial element of $C_n$ form $\ell + 1$ orbits. The equivalency of two such actions will be given by the existence of a conjugating homeomorphism that also satisfies an additional condition on their distinguished fixed points. Two equivalence classes of actions will form a \textit{compatible pair} if the turning angles of their representative actions around their distinguished fixed points add up to $2\pi/n$. The key topological idea in our theory is defining nestled $(n_i,\ell_i)$-actions on the subsurfaces $\tilde{F_i}$ for $i = 1,2$ so that they form a compatible pair, thus giving a root of degree $n = \lcm(n_1,n_2)$.  Conversely, for each root of degree $n$, we reverse this argument to produce a corresponding compatible pair.

In Section~\ref{sec:nestlednl and data sets}, we introduce the abstract notion of a \textit{data set} of degree $n$. As in the case of nonseparating curves, a \textit{data set of degree n} is basically a tuple that encodes the essential algebraic information required to describe a nestled action. We show that equivalence classes of nestled $(n,\ell)$-actions actually correspond to data sets, that is, each class has a corresponding data set representation. Data sets $D_i$ of degree $n_i$, for $i=1,2$ form a \textit{data set pair} $(D_1,D_2)$ when they satisfy the formula $\frac{n}{n_1}k_1 + \frac{n}{n_2}k_2 \equiv 1\bmod n$, where the turning angles at the centers of the disks are $\frac{2\pi k_i}{n_i} \bmod 2\pi$. In Theorem~\ref{thm1}, we show that this number-theoretic condition is an algebraic equivalent of the compatibility condition for actions, thus proving that data set pairs correspond bijectively to conjugacy classes of roots. This theorem is essentially a translation of our topological theory of roots to the algebraic language of data sets.

As an immediate application of Theorem~\ref{thm1}, we show the existence of a root of degree $\lcm(4g_1, 4g_2+2)$, and in Section~\ref{sec:genus2_classification}, we give calculation of roots in low-genus cases. In Section~\ref{sec:spherical}, we obtain some bounds on the orders of spherical nestled actions, that is, nestled actions whose quotient orbifolds are topologically spheres. For example, we prove that all nestled $(n, \ell)$-actions for $n \geq \frac{2}{3}(2g-1)$ have to be spherical. Finally, in Section~\ref{sec:degree-bounds}, we use the main theorem and the results obtained in Section~\ref{sec:spherical} to derive bounds on $n$. We show that in general, $n \leq 4g^2+2g$ and for any positive integer $N$, $n \leq  4g^2 + (4-2N)g + \frac{{(N-2)}^2}{4}$ whenever both $g_i > N+3$.

\section{Nestled $(n,\ell)$-actions}
\label{sec:nestlednl}
An action of a group $G$ on a topological space $X$ is defined as a homomorphism $h : G \rightarrow \Homeo(X)$. Since we are interested only in $C_n$-actions, we will fix a generator $t$ for $C_n$ and identify the action with the isotopy class of the homeomorphism $h(t)$ in $\Mod(X)$. In this section, we introduce nestled $(n,\ell)$-actions and give an example for such an action. These actions will play a crucial role in the theory we will develop for roots of Dehn twists.

\begin{definition}
An orientation-preserving $C_n$-action on a surface $F$ of genus at least 1 is said to be a nestled $(n,\ell)$-action if either $n = 1$, or $n > 1$ and:
\begin{enumerate}
\item [(i)] the action has at least one fixed point,
\item [(ii)] some fixed point has been selected as the distinguished fixed point, and
\item [(iii)] the points fixed by some nontrivial element of $C_n$ form $\ell + 1$ orbits.
\end{enumerate}
This is equivalent to the condition that the quotient orbifold has $\ell + 1$ cone points, one of which is a distinguished cone point of order $n$.
\end{definition}

A nestled $(n, \ell)$-action is said to be \textit{trivial} if $n = 1$, that is, if it is the action of the trivial group on $F$. In this case only, we allow a cone point of order 1 in the quotient orbifold. The distinguished cone point can then be any point in $F$, and we require $\ell=0$.

\begin{definition}
Assume that $F$ has a fixed orientation and fixed Riemannian metric. Let $h$ be a nestled-$(n,\ell)$ action on $F$ with a distinguished fixed point $P$. The \textit{turning angle} $\theta(h)$ for $h$ is the angle of rotation of the induced isomorphism $h_*$ on the tangent space $T_{P}$, in the direction of the chosen orientation.
\end{definition}

\begin{example}[Margalit-Schleimer,~\cite{MS}]
Rotate a regular $(4g+2)$-gon with opposite sides identified about its center $P$ through an angle $\frac{2\pi (g+1)}{(2g + 1)}$. Identifying the opposite sides of $P$, we get a $C_{2g+1}$-action $h$ on $S_{g}$ with three fixed points denoted by $P$, $x$ and $y$. Since the quotient orbifold has three cone points of order $2g + 1$, this defines a nestled $(2g + 1, 2)$-action on ~$S_{g}$. If we choose $P$ as the distinguished fixed point for the action $h$, then $\theta(h) = \frac{2\pi (g+1)}{(2g + 1)}$.
\begin{figure}[h]
\labellist
\small
\pinlabel $P$ [B] at 63 83
\pinlabel $\displaystyle \frac{4\pi}{3}$ [B] at 116 83
\pinlabel $x$ [B] at 81 -8
\pinlabel $y$ [B] at 81 180
\pinlabel $x$ [B] at 0 130
\pinlabel $x$ [B] at 162 130
\pinlabel $y$ [B] at 0 40
\pinlabel $y$ [B] at 162 40
\endlabellist
\centering
\includegraphics[width = 23 ex]{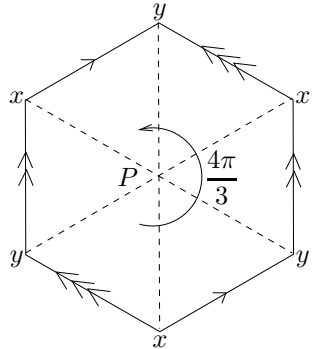}
\caption{A nestled $(2g + 1, 2)$-action for $g = 1$.}
\label{ms}
\end{figure}
\end{example}

\begin{remark}
\label{rem:inv disk}
Every nestled $(n,\ell)$-action has an invariant disk around its distinguished fixed point. Let $F$ be a closed oriented surface with a fixed Riemannian metric $\rho$, and let $h$ be a nestled $(n,\ell)$-action on $F$ with a distinguished fixed point $P$. Consider the Riemannian metric $\bar{\rho}$ defined by
\[ {\langle v, w \rangle}_{\bar{\rho}} = \frac{1}{n} \sum_{i=1}^n {\langle {h^i}_*(v), {h^i}_*(w) \rangle}_{\rho}\ ,\] where $v, w \in T_P F$. Under this metric $\bar{\rho}$, $h$ is an isometry. Since there exists $\epsilon>0$ such that ${exp}_P : B_{\epsilon}(0) \subset T_P F \rightarrow B_{\epsilon}(P) \subset F$ is a diffeomorphism, $h$ preserves the disk $B_{\epsilon}(P)$.
\end{remark}

\begin{definition}
\label{def:nestled eq pair}
Two nestled $(n,\ell)$-actions $h$ and $h'$ on $F$ with distinguished fixed points $P$ and $P'$ are \textit{equivalent} if there exists an orientation-preserving homeomorphism $t\colon F\to F$ such that
\begin{enumerate}
\item[(i)] $t(P)=P'$.
\item[(ii)] $tht^{-1}$ is isotopic to $h'$ relative to $P'$.
\end{enumerate}
\end{definition}

\begin{remark}
By definition, equivalent nestled $(n,\ell)$-actions $h$ and $h'$ on $F$ are conjugate in $\Mod(F)$. Since conjugate homeomorphisms have the same fixed point data, we have that $\theta(h) = \theta(h')$.
\end{remark}


\section{Compatible pairs and roots}
\label{sec:comp-pairs and roots}
Suppose that $C$ is a curve that separates a surface $F$ of genus $g$ into two subsurfaces. As mentioned earlier, the central idea is defining compatible nestled actions on the subsurfaces that ``fit together" to give a degree $n$ root of the Dehn twist $t_C$. We will show in Theorem~\ref{prop:main} that compatible pairs of equivalent actions correspond bijectively to conjugacy classes of roots of $t_C$.

\begin{notation}
Suppose that $C$ separates a closed orientable surface $F$ into subsurfaces of genera $g_1$ and $g_2$, where $g_1 \geq g_2$. Let $F_i$ denote the closed surface obtained by coning the subsurface of genus $g_i$. We will think of $F$ as $(F_1,C)\#(F_2,C)$, that is, the surface obtained by taking the connected sum of the $F_i$ along $C$. For the sake of convenience, we will denote this by $F = F_1 \#_C F_2$.
\end{notation}

\begin{definition}
\label{def:comp-pair}
Equivalence classes $[h_i]$ of nestled $(n_i,\ell_i)$-actions $h_i$ on closed oriented surfaces $F_i$ for $i=1,2$ are said to form a \textit{compatible pair} $([h_1],[h_2])$ if $\theta(h_1)+\theta(h_2)=2\pi/n \bmod 2\pi$.

The integer $n = lcm(n_1,n_2)$ is called the \textit{degree} of the compatible pair. We may treat $([h_1], [h_2])$ to as an unordered pair, since $([h_2], [h_1])$ is a compatible pair if and only if $([h_1], [h_2])$ is.
\end{definition}

\begin{lemma}
\label{lem1}
Let F be a compact orientable surface, possibly disconnected. If $h: F \rightarrow F$ is a homeomorphism such that $h^n$ is isotopic to
$id_F$, then $h$ is isotopic to a homeomorphism $j$ with $j^n = id_F$.
\end{lemma}

\begin{proof}
When $F$ is connected, this is the Nielsen-Kerchkoff theorem \cite{K1, K2, N1}. Suppose that $F$ is not connected. We may asssume that $h$ acts transitively on the set of components $F_1, F_2, ..., F_{\ell}$ of $F$. Choose notation so that $h\mid_{F_i} : F_i \rightarrow F_{i+1}$ and $h\mid_{F_{\ell-1}} : F_{\ell-1} \rightarrow F_1$. Since
$h^n = {(h^l)}^{n/\ell} \simeq {id}_F$, the Nielsen-Kerchkoff theorem implies that ${h^\ell}\mid_{F_1} \simeq j_1$ where $j_1$ is a homeomorphism on $F_1$ with
${j_1}^{n/\ell} = {id}_{F_1}$. Therefore, $id_{F_1} \simeq j_1 \circ {({h^\ell}\mid_{F_1})}^{-1}$ via an isotopy $K_t$. Define an isotopy $H_t$ of $h$ by
$H_t\mid_{F_i} = h$ for $1 \leq i \leq \ell-2$ and $H_{t}\mid_{F_{\ell-1}} = K_t \circ h\mid_{F_{\ell-1}}$.
Then, ${H_1}\mid_{F_{\ell-1}} = K_1 \circ h = j_1 \circ {({h^\ell}\mid_{F_1})^{-1}} \circ h$.
We see that ${(H_1\mid_{F_i})}^\ell = h^i \circ (j_1 \circ h^{1-\ell}) \circ h^{\ell-1-i} = h^i \circ j_1 \circ h^{-i}$ and
${(H_1\mid_{F_i})}^n = {({H_1\mid_{F_i}}^\ell)}^{n/\ell} = h^i \circ {j_1}^{n/\ell} \circ h^{-i} = h^i \circ h^{-i} = {id}_{F_i}$. The required homeomorphism is $j = H_1$.
\end{proof}

\begin{theorem}
Let $F = F_1 \#_C F_2$ be a closed oriented surface of genus $g \geq 2$. Then the conjugacy classes in $\Mod(F)$ of roots of $t_C$ of degree $n$ correspond to the compatible pairs $([h_1],[h_2])$ of equivalence classes of nestled $(n_i,\ell_i)$-actions $h_i$ on $F_i$ of degree $n$.
\label{prop:main}
\end{theorem}

\begin{proof}
We will first prove that every root of degree $n$ yields a compatible pair of $([h_1],[h_2])$ of degree $n$.

Fix a closed annulus neighborhood $N$ of $C$. Let $\widetilde{F}_i$ for $i =1,2$ be the components of $\overline{G-N}$, and denote the genus of $\widetilde {F_i}$ by $g_i$ . We fix coordinates on $F$ so that the subsurface $\widetilde{F}_1$ is to the left of $C$ as shown in Figure~\ref{surf}. By isotopy we may assume that $t_C(C)=C$, $t_C(N)=N$, and $t_C\vert_{\widetilde{F}_i }=id_{\widetilde{F}_i }$ for $i = 1,2$.

\begin{figure}[h]
\labellist
\small
\pinlabel $\widetilde{F}_1$ at 60 90
\pinlabel $\widetilde{F}_2$ at 230 90
\pinlabel ${N}$ at 145 95
\endlabellist
\centering
\includegraphics[width=55 ex]{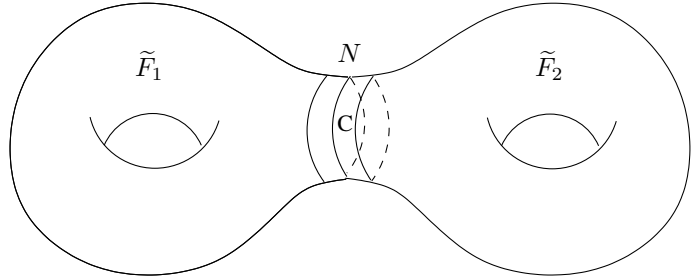}
\caption{The surface $F$ with the separating curve $C$ and the tubular neighborhood $N$ of $C$.}
\label{surf}
\end{figure}

Suppose that $h$ is an $n^{th}$ root of $t_C$. We have $t_C \simeq ht_Ch^{-1} \simeq t_{h(C)}$, which implies that $h(C)$ is isotopic to $C$. Changing $h$ by isotopy, we may assume that $h$ preserves $C$ and takes $N$ to $N$. Put $\widetilde{h_i}=h\vert_{\widetilde {F_i}}$ for $i=1,2$. Since $h^n\simeq t_C$ and both preserve $C$, there is an isotopy from $h^n$ to $t_C$ preserving $C$ and hence one taking $N$ to $N$ at each time. That is, ${\widetilde{h_1}}^n$ is isotopic to $id_{\widetilde{F}_1 }$ and ${\widetilde{h_2}}^n$ is isotopic to $id_{\widetilde{F}_2 }$ . By Lemma~\ref{lem1}, $\widetilde{h_i}$ is isotopic to a homeomorphism whose $n^{th}$ power is $id_{\widetilde{F}_i }$ for $i=1,2$. So we may change $\widetilde{h_i}$ and hence $h$ by isotopy to assume that ${\widetilde{h_i}}^n=id_{\widetilde {F_i}}$ for $i=1,2$.

Let $n_i$ be the smallest positive integer such that ${\widetilde{h_i}}^{n_i}=id_{\widetilde {F_i}}$ for $i = 1,2$. Let $s = lcm(n_1,n_2)$. Clearly, $s \vert n$ since $n_i \vert n$. Also, $h^s = id_{\widetilde{F}_1  \cup \widetilde{F}_2 }$ which implies that $h^s = {{t_C}^d}$ for some integer $d$. Hence, ${(h^s)}^{n/s} = {{({t_C}^d)}^{n/s}}$ i.e. $h^n = {t_C}^{dn/s}$. We get, $t_C = {t_C}^{dn/s}$ which implies that $dn/s = 1$ since no higher power of $t_C$ is isotopic to $t_C$. Hence, $d = 1$ and $n = s = lcm(n_1, n_2)$.

Assume for now that $h$ does not interchange the sides of $C$. We fill in the boundary circles of $\widetilde{F}_1$ and $\widetilde{F}_2 $ with disks to obtain the closed orientable surfaces $F_1$ and $F_2$ with genera $g_1$ and $g_2$ . We then extend $\widetilde {h_i}$ to a homeomorphism $h_i$ on $F_i$ by coning. Thus $h_i$ defines a $C_{n_i}$ action on $F_i$ where $n_i \vert n$, $C_{n_i}=\langle h_i\;\vert\;h_i^{n_i}=1\rangle$ for $i = 1,2$ and $lcm(n_1, n_2) = n$. Since the homeomorphism $h_i$ fixes the center point $P_i$ of the disk $\overline{F_i - \widetilde{F}_i }$, we choose $P_i$ as the distinguished fixed point for $h_i$. So $h_i$ defines a nestled $(n_i,\ell_i)$-action on $F_i$ for some $\ell_i$.

The orientation on $F$ restricts to orientations on the $F_i$, so that we may speak of rotation angles $\theta(h_i)$ for $h_i$. Then the rotation angle $\theta(h_i) = 2\pi k_i/n_i$ for some $k_i$ with $\gcd(k_i,n_i)=1$. As seen in Figure~\ref{fig:twist}, the difference in turning angles equals $2\pi k_2/n_2-(-2\pi k_1/n_1)=2\pi/n$, giving $\theta(h_1) + \theta(h_2) \equiv 2\pi/n \bmod 2\pi$. That is, $(h_1,h_2)$ is a compatible pair.

\begin{figure}[h]
\labellist
\small
\pinlabel $A$ [B] at 160 95
\pinlabel $P_1$ [B] at 73 80
\pinlabel $h_1(A)$ [B] at -13 142
\pinlabel $B$ [B] at 358 95
\pinlabel $h_2(B)$ [B] at 250 10
\pinlabel $P_2$ [B] at 260 105
\pinlabel $A$ [B] at 555 175
\pinlabel $B$ [B] at 555 30
\endlabellist
\centering
\includegraphics[width=60 ex]{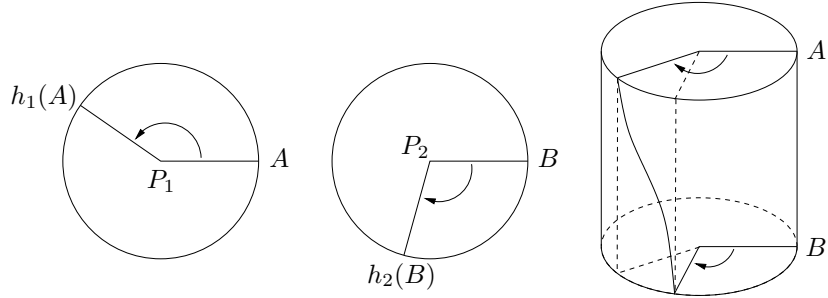}
\caption{The local effect of $h_1$ and $h_2$ on disk neighborhoods of $P_1$ and $P_2$ in $F_1$ and $F_2$, and the effect of $h$ on the neighborhood $N$ of $C$ in $F$. Only the boundaries of the disk neighborhoods are contained in $\widetilde{F}_i$, where they form the boundary of $N$. The rotation angle $\theta(h_1)$ is $2\pi {k_1}/{n_1}$ and the angle $\theta(h_2)$ is $2\pi k_2/n_2 = 2\pi (1/n - k_1/n_1)$.}
\label{fig:twist}
\end{figure}

Suppose now that $h$ interchanges the sides of $C$. Then $h$ must be of even order, say $2n$, and $h^2$ preserves the sides of $C$ and is of order $n$. Since the actions of $h^2\vert_{\widetilde{F_i}}$ on the $\widetilde{F_i}$ are conjugate by $h\vert_{\widetilde{F_1}\cup \widetilde{F_2}}$, these actions will induce conjugate $C_n$-actions on the coned surfaces $F_i$. Consequently, these induced actions will have the same turning angles at the centers $P_i$ of the coned disks of $F_i$. For this compatible pair of nestled $(n_i,\ell_i)$-actions, say $(h_1,h_2)$, associated with $h^2$, we must have $\theta(h_1) = \theta(h_2) = \pi/n$ and $n_1 = n_2 = n$. If we extend to $N$ using a simple left-handed twist, the twisting angle is $2\pi k/n$, and consequently $h^{2n}=t_{C}^{2k}$. Other extensions will differ from this by full twists, giving $h^{2n}=t_{C}^{2k+2jn}$ for some integer~$j$. In any case, $h^{2n}$ cannot equal $t_{C}$. This proves that $h$ cannot reverse the sides of $C$.

Suppose that we have roots $h$ and $h'$ that are conjugate in $\Mod(F)$, that is, there exists $t \in \Mod(F)$ such that $h' = t\circ h \circ t^{-1}$. Then ${(h')}^n = t\circ h^n \circ t^{-1}$, that is, $t_C = t\circ t_C \circ t^{-1} = t_{t(C)}$. This shows that $C$ and $t(C)$ are isotopic curves. Changing $t$ by isotopy, we may assume that $t(C) = C$ and $t(N)=N$. Let $t_i$, $h_i$ and $h_i'$ respectively denote the extensions of $t\vert_{\widetilde{F_i}}$, $h\vert_{\widetilde{F_i}}$ and $h'\vert_{\widetilde{F_i}}$ to $F_i$ by coning.

Assume for now that $t$ does not exchange the sides of $C$. Since $t$, $h$ and $h'$ all preserve $N$, we may assume that the isotopy from $t\circ h\circ t^{-1}$ to $h'$ preserves $N$, and consequently each $t_i\circ h_i\circ {t_i}^{-1}$ is isotopic to $h_i'$ preserving $P_i$. Since $t_i$ takes $P_i$ to $P_i$, $h_i$ and $h_i'$ are equivalent as nestled $(n_i,\ell_i)$-actions on $F_i$, so $h$ and $h'$ produce the same compatible pair $([h_1],[h_2])$.

Suppose that $t$ exchanges the sides of $C$. Then $g_1 = g_2$,  $h_{3-i}' \simeq t_i\circ h_i\circ {t_i}^{-1}$ and $t_i(P_i) = P_{3-i}$. So the actions $h_1$ and $h_2'$ are equivalent, as are actions $h_1'$ and $h_2$. Therefore, the (unordered) compatible pairs for the two roots are the same.

Conversely, given a compatible pair $([h_1],[h_2])$ of equivalence classes of nestled $(n_i,\ell_i)$-actions, we can reverse the argument to produce a root $h$. For let $P_i$ denote the distinguished fixed point of $h_i$ and let $p_i$ denote the corresponding cone point of order $n_i$ in the quotient orbifold $\O_i$. By Remark~\ref{rem:inv disk}, there exists an invariant disk $D_i$ for $h_i$ around $p_i$. Removing $D_i$ produces the surfaces $\widetilde{F}_i $, and attaching an annulus $N$ produces the surface $F$ of genus $g$. Condition (ii) on compatible pairs ensures that the rotation angles work correctly to allow an extension of $h_1\vert_{\widetilde{F}_1 }\cup h_2\vert_{\widetilde{F}_2 }$ to an $h$ with $h^n$ being a single Dehn twist about $C$.

It remains to show that the resulting root $h$ of $t_C$ is determined up to conjugacy in the mapping class group of $F$. Suppose that $h_i' \in [h_i]$. Let $P_i'$ denote the distinguished fixed point for $h_i'$, and let $D_i'$ be an invariant disk for $h_i'$ around $P_i'$. Removing the $D_i's$ produces surfaces $\widetilde{F}_i' \cong F_i$, for $i = 1,2$, and attaching an annulus $N'$ with a $1/n^{th}$ twist, extends $h_1'\vert_{\widetilde{F}_1'}\cup h_2'\vert_{\widetilde{F}_2'}$ to a homeomorphism $h'$ on a surface $F' \cong F$ of genus $g$. Since $h_i' \in [h_i]$, by definition, there exists $t_i$ such that $t_i(P_i) = P_i'$ and $t_i\circ h_i\circ {t_i}^{-1} \simeq h_i'$ rel $P_i'$ via an isotopy $H_i$ in $\Mod(F_i')$. Since $h_i$ and $h_i'$ have finite order and are conjugate up to isotopy by $t_i$, we may assume that $t_i(D_i) = D_i'$ and, identifying $F$ and $F'$ using $t$, that the isotopy $H_i$ from $t_i\circ h_i\circ{t_i}^{-1}$ to $h_i'$ is relative to $D_i$. With respect to this identification, we choose a $k:N \rightarrow N$ such that $h'\vert_{N} = k\circ h\vert_{N}\circ k^{-1}$. Now define $t:F\rightarrow F$ by $t\vert_{\widetilde{F_i}} = h_i\vert_{\widetilde{F_i}}$, and $t\vert_{N} = k$. Then $h' \simeq t\circ h \circ t^{-1}$ via an isotopy $H$ given by $H\vert_{\widetilde{F_i}} = H_i\vert_{\widetilde{F_i}}$, and $H\vert_{N} = id_{N}$.
\end{proof}

\section{Nestled $(n,\ell)$-actions and data sets}
\label{sec:nestlednl and data sets}
In this section, we introduce the language of data sets of degree $n$ in order to algebraically encode classes of nestled $(n,\ell)$-actions. We will also prove that equivalence classes of nestled $(n,\ell)$-actions actually correspond to data sets.

\begin{definition}
A \textit{data set} for $F$ is a tuple
$D = (n, \widetilde{g}, a; (c_1,x_1),\ldots,(c_{\ell}, x_{\ell}))$ where $n$, $\widetilde{g}$ and the $x_i$ are integers, $a$ is a residue class modulo $n$, and each $c_i$ is a residue class modulo $x_i$, such that
\begin{enumerate}
\item[(i)] $n \geq 1$, $\widetilde{g}\geq 0$, each $x_i>1$, and each $x_i$ divides $n$.
\item[(ii)] $\gcd(a,n) = \gcd(c_i,x_i) = 1$.
\item[(iii)] $a + \displaystyle\sum_{i=1}^{\ell} \frac{n}{x_i}c_i \equiv 0\bmod n$.
\end{enumerate}
The number $n$ is called the \textit{degree} of the data set. If $n = 1$, then we require that $a = 1$, and the data set is $D = (1, \widetilde{g}, 1;)$. The integer $g$ defined by \[g = \widetilde{g}n + \frac{1}{2}(1 - n) + \frac{1}{2}\displaystyle\sum_{i=1}^{\ell} \frac{n}{x_i}(x_i - 1)\] is called the \textit{genus} of the data set. We consider two data sets to be the same if they differ by reordering the pairs $(c_1,x_1),\ldots,(c_{\ell}, x_{\ell})$.
\end{definition}

\begin{remark}
\label{prop1}
For any data set $D = (n, \widetilde{g}, a; (c_1,x_1),\ldots,(c_{\ell}, x_{\ell}))$, \\$\lcm\{x_1,x_2,\ldots,x_n\} = n$. To see this, put $k = lcm(x_{1}, x_{2}, \ldots, x_{\ell})$. Since each $x_i\mid n$, $k\mid n$. So it remains to show that $n\mid k$. Condition $(iii)$ implies that
\[\frac{ak}{k} + \displaystyle\sum_{i=1}^{\ell}\frac{n(k/x_i)}{k}c_i \equiv 0\bmod n\ .\]
\noindent Multiplying by $k$ we get
\[ak + n\displaystyle\sum_{i=1}^{\ell}(k/x_i)c_i \equiv 0\bmod n\ .\]
\noindent Since $gcd(a, n) = 1$, we have $n\mid k$.
\end{remark}

We will prove in the following proposition that data sets of degree $n$ correspond to equivalence classes nestled-$(n,\ell)$ actions.

\begin{proposition}
\label{prop:nestled-data set}
Data sets of degree $n$ and genus $g$ correspond to equivalence classes of nestled $(n,\ell)$-actions on closed orientable surfaces of genus ~$g$.
\end{proposition}

\begin{proof}
Let $h$ be a nestled-$(n,\ell)$ action on a closed orientable surface $F$ of genus $g$. Let $\O$ be the quotient orbifold for the action and let $\widetilde{g}$ be the genus of its underlying $2$-manifold. Let $P$ be the distinguished fixed point of $h$ and let $p$ be the cone point in $\O$ of order $n$ that is its image in $\O$. Let $p_1,\ldots\,$, $p_{\ell}$ be the other possible cone points of $\O$, if any.

\begin{figure}[h]
\labellist
\small
\pinlabel $p$ [B] at 568 75
\pinlabel $p_1$ [B] at 525 97
\pinlabel $p_2$ [B] at 463 102
\pinlabel $\alpha$ [B] at 608 105
\pinlabel $\gamma_1$ [B] at 550 142
\pinlabel $\gamma_2$ [B] at 463 150
\endlabellist
\centering
\includegraphics[width=65 ex]{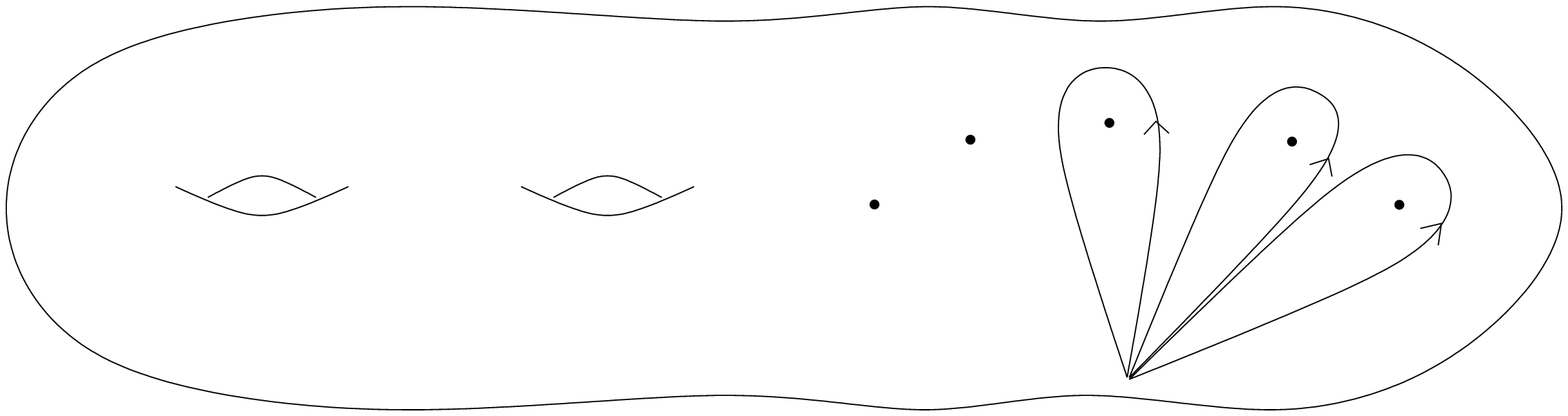}
\caption{The orbifold $\O$}
\label{fig:orb}
\end{figure}

Figure~\ref{fig:orb} shows a generator $\alpha$ of the orbifold fundamental group $\pi_1^{orb}(\O)$ that goes around the point $p$, and generators $\gamma_{i}, 1 \leq i \leq \ell$ going around $p_i$. Let $a_i$ and $b_i$, $1\leq j\leq \widetilde{g}$ be standard generators of the ``surface part'' of $\O$, chosen to give the following presentation of $\pi_1^{orb}(\O)$:

\begin{gather*}
\pi_1^{orb}(\O_i)=\langle \alpha, \gamma_1,\ldots, \gamma_{\ell},
a_1,b_1,\ldots, a_{\widetilde{g}}, b_{\widetilde{g}}\;\vert\;\\
{\alpha}^{n}=\gamma_1^{x_1}=\cdots =\gamma_i^{x_i}=1,\;
\alpha \gamma_1 \cdots \gamma_{\ell}=\prod_{1=1}^{\widetilde{g}}[a_i,b_i]\;\rangle.
\end{gather*}

From orbifold covering space theory ~\cite{T1}, we have the following exact sequence:
\[ 1 \longrightarrow \pi_1(F) \longrightarrow \pi_1^{orb}(\O)
\stackrel{\rho}{\longrightarrow} C_n \longrightarrow 1\ .\]

\noindent The homomorphism $\rho$ is obtained by lifting path representatives of elements of $\pi_1^{orb}(\O)$--- these do not pass through the cone points so the lifts are uniquely determined.

For $1\leq i\leq l$, the preimage of $p_i$ consists of $n/x_i$ points cyclically permuted by $h$, where $x_i$ is the order of the stabilizer of each point in the preimage of $p_i$. Each of the points has stabilizer generated by $h^{n/x_i}$. Its rotation angles must be the same at all points of the orbit, since its action at one point is conjugate by a power of $h$ to its action at each other point. So the rotation angle at each point is of the form $2\pi c_i'/x_i$, where $c_i'$ is a residue class modulo $x_i$ and $\gcd(c_i',x_i) = 1$. Lifting the $\gamma_i$, we have that $\rho_1(\gamma_i)= {h}^{(n/x_i)c_i}$ where $c_ic_i' \equiv 1 \bmod x_i$.

Finally, we have  $\rho(\prod_{i=1}^{\widetilde {g}}[a_i,b_i])=1$, since $C_n$ is abelian, so
\[ 1=\rho_i(\alpha\gamma_1\cdots \gamma_{\ell})= t^{a+(n/x_1)c_1+\cdots+ (n/x_i)c_i} \]
\noindent
giving \[a + \displaystyle\sum_{i=1}^{\ell} \frac{n}{x_i}c_i \equiv 0\bmod n\ .\]

The fact that the data set $D$ has genus equal to $g$ follows easily from the multiplicativity of the orbifold Euler characteristic for the orbifold covering $F\to \O$:
\begin{equation}
\frac{2-2g}{n} = 2 - 2\widetilde{g} + \left(\frac{1}{n}-1\right) + \sum_{i=1}^{\ell} \left(\frac{1}{x_i}-1\right)
\label{chi:orb}
\end{equation}
\noindent Thus, $h$ gives a data set $D = (n, \widetilde{g}, a; (c_1,x_1),\ldots,(c_{\ell},x_{\ell}))$ of degree $n$ and genus $g$.

Consider another nestled $(n,\ell)$-action $h'$ in the equivalence class of $h$ with a distinguished fixed point $P'$. Then by definition there exists an orientation-preserving homeomorphism $t \in \Mod(F)$ such that $t(P)=P'$ and $th't^{-1}$ is isotopic to $h$ relative to $P$. Therefore, the two actions will have the same fixed point data and hence produce the same data set $D$.

Conversely, given a data $D=(n, \widetilde{g}, a; (c_1,x_1),\ldots,(c_{\ell},x_{\ell}))$, we can reverse the argument to produce an equivalence class of a nestled $(n,\ell)$-action $h$ on a surface $F$ of genus $g$. We construct the orbifold $\O$ and representation $\rho\colon \pi_1^{orb}(\O)\to C_n$. Any finite subgroup of $\pi_1^{orb}(\O)$ is conjugate to one of the cyclic subgroups generated by $\alpha$ or a $\gamma_i$, so condition (ii) in the definition of the data set ensures that the kernel of $\rho$ is torsionfree. Therefore the orbifold covering $F \to \O$ corresponding to the kernel is a manifold, and calculation of the Euler characteristic shows that $F$ has genus $g$.

It remains to show that the resulting action on $F$ is determined up to our equivalence in $\Mod(F)$. Suppose that two actions $h$ and $h'$ on $F$ with distinguished fixed points $P$ and $P'$ have the same data set $D$. $D$ encodes the fixed-point data of the periodic transformations $h$. By a result of J. Nielsen~\cite{N1} (see also A. Edmonds~\cite[Theorem 1.3]{AE}), $h$ and $h'$ have to be conjugate by an orientation-preserving homeomorphism $t$. As in the proof of Theorem 1.1 in~\cite{MK1}, $t$ may be chosen so that it preserves $t(P)=P'$. Thus $D$ determines $h$ up to equivalence.
\end{proof}

Proposition~\ref{prop:nestled-data set} enables us to view equivalence classes of nestled $(n,\ell)$-actions simply as data sets.

\begin{notation}
We will denote a data set of degree $n$ and genus $g$ by $D_{n,g,i}$, where $i$ is an index. The trivial data set $D = \{1,g,1;\}$, for any $g$, will be denoted by $D_{1,g}$.
\end{notation}

\begin{example}
\label{ex:nestled actions}
For every $g \geq 1$, below are examples of data sets that represent nestled $(n,2)$-actions, when $n$ is $2g+1$, $4g$ and $4g+2$:
\begin{enumerate}
\item[(i)] $D_{2g+1,g,1} = (2g + 1, 0, 1; (g, 2g+1), (g, 2g + 1))$.
\item [(ii)] $D_{4g,g,1} = (4g, 0, 1; (1, 2), (2g - 1, 4g))$.
\item [(iii)] $D_{4g+2,g,1} = (4g+2, 0, 1; (1, 2), (g, 2g + 1))$.
\end{enumerate}
\end{example}

\begin{remark}
\label{chi:orb1}
For the data set $D = (n, \widetilde{g}, a; (c_1,x_1),\ldots,(c_n, x_{\ell}))$ associated with a nestled $(n, \ell)$-action, Equation~\ref{chi:orb} in the proof of Proposition~\ref{prop:nestled-data set} gives the following inequality
\begin{equation}
\frac{1-2g}{n} = -(\ell - 1) - 2\widetilde{g} + \sum_{i=1}^{\ell} \frac{1}{x_i} \leq -(\ell - 1) + \sum_{i=1}^{\ell} \frac{1}{x_i}\ .\
\label{eq:nestled_action_riemann_hurwitz}
\end{equation}
\end{remark}

Let $\O$ be the quotient orbifold for a nestled $(n, \ell)$-action. Let $\alpha$ be a generator of $\O$ going around the distinguished order $n$ cone point and let $\gamma_1, \gamma_2, \ldots, \gamma_{\ell}$ be generators going around the other cone points. We have the exact sequence
\[ 1 \longrightarrow \pi_1(F) \longrightarrow \pi_1^{orb}(\O)\stackrel{\rho}{\longrightarrow} C_{n} \longrightarrow 1\ .\]

\begin{remark}
There exists no non-trivial action with $\ell = 0$. Suppose that we assume the contrary. Then $\O$ has a distinguished cone point of order $n$ and no
other cone points. Let $a_j$ and $b_j$, $1\leq j\leq \widetilde{g}$ be the standard generators of the ``surface part'' of $\O$. Then, the fundamental
group of $\O$ has the following representation
\[\pi_1^{orb}(\O)=\langle  \alpha, a_{1},b_{1},\ldots, a_{\widetilde{g}}, b_{\widetilde{g}}\;\vert{\alpha}^n = 1, \alpha = \prod_{j=1}^{\widetilde{g}}[a_j,b_j]\;\rangle\ .\]
Since $C_n$ is abelian, $\rho(\alpha)=\rho(\prod_{j=1}^ {\widetilde{g}}[a_j,b_j]) = 1$, which is impossible since $\rho$ has torsion free kernel.
\label{rk:no_ell_equal_zero_nestled_actions}
\end{remark}

\section{Data set pairs and roots}
\label{sec:data set pair-roots}

By Theorem~\ref{prop:main}, each conjugacy class of a root of $t_C$ in $\Mod(F)$ corresponds to a compatible pair $([h_1],[h_2])$ of (equivalence classes of) nestled actions, and by Proposition~\ref{prop:nestled-data set}, such a pair determines a pair $(D_1,D_2)$ of data sets. To determines which pairs arise, we must replace the geometric compatibility condition in Theorem~\ref{prop:main} by an algebraic compatibility condition on the corresponding data sets.

\begin{definition}
\label{def:dspair}
Two data sets $D_1 = (n_1, \widetilde{g_1}, a_1; (c_{11},x_{11}),\ldots,(c_{1\ell},x_{1\ell}))$ and $D_2 = (n_2, \widetilde{g_2}, a_2; (c_{21},x_{21}),\ldots,(c_{2m},x_{2m}))$ are said to form a \textit{data set pair} $(D_1, D_2)$ if
\begin{equation}
\label{dspair}
\frac{n}{n_1}k_1 + \frac{n}{n_2}k_2 \equiv 1\bmod n
\end{equation}
where $n = \lcm(n_1, n_2)$ and $a_ik_i \equiv 1 \bmod n_i$. Note that although the $k_i$ are only defined modulo $n_i$, the expressions $\frac{n}{n_i}k_i$ are
well-defined modulo $n$. The integer $n$ is called the \textit{degree} of the data set pair and $g = g_1 + g_2$ is called the \textit{genus} of the data set pair. We consider $(D_1, D_2)$ to be an unordered pair, that is, $(D_1, D_2)$ and $(D_2, D_1)$ are equivalent as compatible pairs.
\end{definition}

We can now reformulate Theorem~\ref{prop:main} in terms of data sets.

\begin{theorem}
Let $F = F_1 \#_C F_2$ be a closed oriented surface of genus $g \geq 2$. Then, data set pairs $(D_1,D_2)$ of degree $n$ and genus $g$, where $D_1$ is a data set of genus $g_1$ and $D_2$ is a data set of genus $g_2$, correspond to the conjugacy classes in $Mod(F)$ of roots of $t_C$ of degree $n$.
\label{thm1}
\end{theorem}

\begin{proof}
Let $h$ denote the conjugacy class of a root of $t_C$ of degree $n$ with compatible pair representation $([h_1],[h_2])$. From Proposition~\ref{prop:nestled-data set}, the $h_i$ correspond to data sets $D_i = (n_i, \widetilde{g_i}, a_i; (c_{i1},x_{i1}),\ldots,(c_{i\ell_i},x_{1\ell_i}))$. So it suffices to show that the geometric condition $\theta(h_1) + \theta(h_2) = 2\pi/n$ in Definition~\ref{def:comp-pair} is equivalent to the condition $\frac{n}{n_1}k_1 + \frac{n}{n_2}k_2 \equiv 1\bmod n$ in Definition~\ref{def:dspair}.

As in the proof of Proposition~\ref{prop:main}, let $P_i$ denote the center of the filling disk of the subsurface $\widetilde{F_i}$ of genus $g_i$. Choosing $P_i$ as the distinguished fixed point of $h_i$, we get that $\theta(h_i) = 2\pi k_i/n_i$, where $\gcd(k_i, n_i) = 1$ and $a_ik_i \equiv 1 \bmod n_i$. Since ${h}^n=t_{C}$, the left-hand twisting angle along $N$ is $2\pi/n$, which equals $2\pi k_2/n_2-(-2\pi k_1/n_1)=2\pi/n$, giving $\frac{n}{n_1}k_1 + \frac{n}{n_2}k_2 \equiv 1 \bmod n$. The converse is just a matter of reversing the argument.
\end{proof}

\begin{corollary}
Suppose that $F = F_1 \#_C F_2$. Then there always exists a root of the Dehn twist $t_C$ about $C$ of degree $lcm(4g_1, 4g_2 + 2)$.
\label{coro:rootorder}
\end{corollary}

\begin{proof}
As in Theorem~\ref{thm1}, let $\widetilde{F}_i $ denote the subsurfaces obtained by cutting $F$ along $C$, and let $F_i$ denote the surfaces
obtained by adding disks to the $F_i$. Let $n_1 = 4g_1$ and $n_2 = 4g_2 +  2$. From Example~\ref{ex:nestled actions}, for any residue class $a_i$ modulo $n_i$ with $\gcd (a_i, n_i) = 1$, the data set $D_1 = (n_1, 0, a_1; (-a_1, 2g_1), (a_1, 4g_1))$ defines a nestled $(n_1, 2)$-action on a surface $F_1$ of genus $g_1$, and the data
set $D_2 =(n_2, 0, a_2; (a_2, 2), (a_2g_2, 2g_2 + 1))$ defines a nestled $(n_2, 2)$-action on $F_2$ of genus $g_2$.

Let $k_i$ denote the inverse of $a_i$ modulo $n_i$ and let $n = \lcm(n_1, n_2)$. We will now show that the $a_i$ can be selected so that Equation~\ref{dspair} is satisfied. In other words, this will prove that $D_1$ and $D_2$ form a data set pair $(D_1, D_2)$. Since $\frac{n}{n_1}$ and $\frac{n}{n_2}$ are relatively prime, there always exist integers $p$ and $q$ such that
\[\frac{n}{n_1}p + \frac{n}{n_2}q = 1\ .\]

In particular, since $\frac{n}{n_1}$ and $\frac{n}{n_2}$ are not both odd, by \cite[Lemma 7.1]{MK1}, $p$ and $q$ can be chosen so that $gcd(p, n_1) = gcd(q, n_2) = 1$. Let $k_1$ be the residue class of $p$ modulo $n_1$ and let $k_2$ be the residue class of $q$ modulo $n_2$. Taking modulo $n$, we get
\[\frac{n}{n_1}k_1 + \frac{n}{n_2}k_2 \equiv 1\bmod n\ .\]
Therefore, by Theorem~\ref{thm1}, there exists a root of $t_{C}$ of order $lcm(4g_1, 4g_2 + 2)$.
\end{proof}

\begin{corollary}
Let $F = F_1 \#_C F_2$ be a closed oriented surface of genus $g \geq 2$. Suppose that $M$ denotes the maximum degree of a root of the Dehn twist $t_C$ about $C$. Then $2g^2+2g \leq M$.
\end{corollary}

\begin{proof}
If $g$ is even, then Corollary~\ref{coro:rootorder} with $g_1 = g_2 = \frac{g}{2}$ gives a root of degree $\lcm(2g,2g+1) = 2g(2g+1)$. If $g$ is odd, then $g_1 = \frac{g+1}{2}$ and $g_2 = \frac{g-1}{2}$ gives a root of degree $\lcm(2(g+1),2g) \geq 2g(g+1)$.
\end{proof}

\section{Classification of roots for the closed orientable surfaces of genus 2 and 3}

\label{sec:genus2_classification}

\subsection{Surface of genus 2} Let $F$ denote the closed orientable surface of genus 2. Up to homeomorphism, there is a unique curve $C$ that separates $F$ into two subsurfaces of genus 1. Given a root of $t_C$, the process described in the proof of Theorem~\ref{thm1} produces orientation-preserving $C_{n_i}$ actions on the tori $F_i$ for $i = 1, 2$ with $n = lcm(n_1, n_2)$.

 If a cyclic group $C_n$ acts faithfully on a surface $F$ fixing a point $x_0$, then the map $C_n \longrightarrow Aut(\pi_1(F, x_0))$ is a monomorphism \cite[Theorem 2, p.43]{F1}. We also know that the group of orientation-preserving automorphisms $Aut^+(\pi_1(F_i, x_0)) \cong  Aut^+(\mathbb{Z} \times \mathbb{Z}) \cong SL(2, \mathbb{Z}) \cong \mathbb{Z}_4 \ast_{\mathbb{Z}_2} \mathbb{Z}_6$. Since any element of finite order of an amalgamated product $A \ast_C B$ is conjugate into one of the groups $A$ or $B$ ~\cite{MKS}, it can only be of order 2, 3, 4 or 6. Taking the least common multiple of any two of these orders gives 12 as the only other possibility for the order of a root of $t_C$. We summarize these inferences in the following corollary.

\begin{corollary}
Let $F$ be the closed orientable surface of genus 2 and $C$ a separating curve in $F$. Then a root of a Dehn twist $t_C$ about $C$ can only be of degree 2, 3, 4, 6, or 12.
\end{corollary}

Given below are the data set pairs that represent each conjugacy class of roots.

\noindent For $n = 2$:
\begin{enumerate}
\item[(i)] $(D_{2,1,1}, D_{1,1})$, where $D_{2,1,1} = (2, 0, 1; (1, 2), (1, 2), (1, 2))$.
\end{enumerate}
For $n = 3$:
\begin{enumerate}
\item[(i)] $(D_{3,1,1}, D_{1,1})$, where $D_{3,1,1} = (3, 0, 1; (1, 3), (1, 3))$.
\item[(ii)] $(D_{3,1,2}, D_{3,1,2})$, where $D_{3,1,2} = (3, 0, 2; (2, 3), (2, 3))$.
\end{enumerate}
For $n = 4$:
\begin{enumerate}
\item[(i)] $(D_{4,1,1}, D_{1,1})$, where $D_{4,1,1} = (4, 0, 1; (1, 2), (1, 4))$.
\item[(ii)] $(D_{4,1,2}, D_{2,1,1})$, where $D_{4,1,2} = (4, 0, 3; (1, 2), (3, 4))$.
\end{enumerate}
For $n = 6$:
\begin{enumerate}
\item[(i)] $(D_{6,1,1}, D_{1,1})$, where $D_{6,1,1} = (6, 0, 1; (1, 2), (1, 3))$.
\item[(ii)] $(D_{6,1,2}, D_{3,1,1})$, where $D_{6,1,2} = (6, 0, 5; (1, 2), (2, 3))$.
\item[(iii)] $(D_{3,1,2}, D_{2,1,1})$.
\end{enumerate}
For $n = 12$:
\begin{enumerate}
\item[(i)] $(D_{6,1,2}, D_{4,1,1})$.
\item[(ii)] $(D_{4,1,2}, D_{3,1,1})$.
\end{enumerate}

\noindent It can be shown using elementary calculations that these are the only possible roots for the various orders. For example, when $n = 12$, the condition $\lcm(n_1,n_2) = 12$ would imply that the set $\{n_1, n_2\}$ can be either $\{6, 4\}$ or $\{4, 3\}$. When $n_1 = 6$ and $n_2 = 4$, the data set pair condition gives $2k_1 + 3k_2 \equiv 1 \bmod 12$. Since $k_i$ is a residue modulo $n_i$, the only possible solution to this equation is $k_1 = 5$ and $k_2 = 1$. This would imply that $a_1 = 5$ and $a_2 = 1$ since $a_i$ is the inverse of $k_i$ modulo $n_i$. Geometrically, this represents the root $h$ of $t_C$ whose twisting angle on one side is $2\pi k_1/n_1 = 5\pi /3$ and on the other side of $C$ is $2\pi k_2/n_2 = \pi/2$. Each data set $D_i$ in the data set pair $(D_1, D_2)$ is then uniquely determined by condition $(iii)$ (for data sets) and the formula for calculating the genus $g_i$. Similar calculations can be used to determine all the data set pairs for the surface of genus 3.

\subsection{Surface of genus 3} Up to homeomorphism, the surface of genus $g = 3$ has a unique curve that separates the surface into two subsurfaces of genera 2 and 1.

Given below are the data set pairs that represent roots of various degrees.
For $n=2$:
\begin{enumerate}
\item[(i)] $(D_{1,2},D_{2,1,1})$.
\item[(ii)] $(D_{2,2,1},D_{1,1})$, where $D_{2,2,1} = (2,0,1;(1,2),(1,2),(1,2),(1,2),(1,2))$.
\item[(iii)] $(D_{2,2,2},D_{1,1})$, where $D_{2,2,2} = (2,1,1;(1,2))$.
\end{enumerate}
For $n=3$:
\begin{enumerate}
\item[(i)] $(D_{1,2},D_{3,1,1})$.
\item[(ii)] $(D_{3,2,1},D_{1,1})$, where $D_{3,2,1} = (3,0,1;(2,3),(2,3),(1,3))$.
\item[(iii)] $(D_{3,2,2},D_{1,1})$, where $D_{3,2,2} = (3,0,2;(1,3),(1,3),(2,3))$.
\end{enumerate}
For $n=4$:
\begin{enumerate}
\item[(i)] $(D_{1,2},D_{4,1,1})$.
\item[(ii)] $(D_{4,2,1},D_{1,1})$, where $D_{4,2,1} = (4,0,1;(1,2),(1,2),(3,4))$.
\item[(iii)] $(D_{4,2,2},D_{4,1,1})$, where $D_{4,2,2} = (4,0,3;(1,2),(1,2),(2,4))$.
\end{enumerate}
For $n=5$:
\begin{enumerate}
\item[(i)] $(D_{5,2,1},D_{1,1})$, where $D_{5,2,1} = (5,0,1;(1,5),(3,5))$.
\item[(ii)] $(D_{5,2,2},D_{1,1})$, where $D_{5,2,2} = (5,0,1;(2,5),(2,5))$.
\end{enumerate}
For $n=6$:
\begin{enumerate}
\item[(i)] $(D_{1,2},D_{6,1,2})$.
\item[(ii)] $(D_{6,2,1},D_{1,1})$, where $D_{6,2,1} = (6,0,1;(2,3),(1,6))$.
\item[(iii)] $(D_{2,2,1},D_{3,1,2})$.
\item[(iv)] $(D_{2,2,2},D_{3,1,2})$.
\item[(v)] $(D_{3,2,2},D_{2,1,1})$.
\item[(vi)] $(D_{3,2,1},D_{6,1,2})$.
\item[(vii)] ($D_{6,2,2}, D_{3,1,1})$, where $D_{6,2,2} = (6,0,5;(1,3),(5,6))$.
\end{enumerate}
For $n=8$:
\begin{enumerate}
\item[(i)] $(D_{8,2,1},D_{1,1})$, where $D_{8,2,1} = (8,0,1;(1,2),(3,8))$.
\item[(ii)] $(D_{8,2,2},D_{2,1,1})$, where $D_{8,2,2} = (8,0,5;(1,2),(7,8))$.
\item[(iii)] $(D_{8,2,3},D_{4,1,1})$, where $D_{8,2,3} = (8,0,7;(1,2),(5,8))$.
\item[(iv)] $(D_{8,2,4},D_{4,1,2})$, where $D_{8,2,4} = (8,0,3;(1,2),(1,8))$.
\end{enumerate}
For $n=10$:
\begin{enumerate}
\item[(i)] $(D_{10,2,1}, D_{1,1})$, where $D_{10,2,1} = (10,0,1;(1,2),(2,5))$.
\item[(ii)] $(D_{5,2,3}, D_{2,1,1})$, where $D_{5,2,3} = (5,0,3;(1,5),(1,5))$.
\item[(iii)] $(D_{5,2,4}, D_{2,1,1})$, where $D_{5,2,4} = (5,0,3;(3,5),(4,5))$.
\end{enumerate}
For $n=12$:
\begin{itemize}
\item[(i)] $(D_{4,2,2},D_{3,1,1})$.
\item[(ii)] $(D_{3,2,1},D_{4,1,2})$.
\item[(iii)] $(D_{4,2,1},D_{6,1,2})$.
\item[(iv)] $(D_{6,2,2}, D_{4,1,1})$.
\end{itemize}
For $n=15$:
\begin{itemize}
\item[(i)] $(D_{5,2,5}, D_{3,1,2})$, where $D_{5,2,5}= (5,0,3;(1,5),(1,5))$.
\item[(ii)] $(D_{5,2,6}, D_{3,1,2})$, where $D_{5,2,6}= (5,0,3;(3,5),(4,5))$.
\end{itemize}
For $n=20$:
\begin{itemize}
\item[(i)] $(D_{5,2,5},D_{4,1,1})$, where $D_{5,2,5} = (5,0,4;(4,5),(2,5))$.
\item[(ii)] $(D_{5,2,6},D_{4,1,1})$, where $D_{5,2,6} = (5,0,4;(3,5),(3,5))$.
\item[(iii)] $(D_{10,2,1}, D_{4,1,2})$, where $D_{10,2,1} = (10,0,7;(1,2),(4,5))$.
\end{itemize}
For $n=24$:
\begin{itemize}
\item[(i)] $(D_{8,2,4}, D_{3,1,2})$.
\item[(ii)] $(D_{8,2,3}, D_{6,1,1})$.
\end{itemize}
For $n = 30$:
\begin{itemize}
\item[(i)] $(D_{10,2,2},D_{3,1,1})$, where $D_{10,2,2} = (10,0,9;(1,2),(3,5))$.
\item[(ii)] $(D_{5,2,7},D_{6,1,2})$, where $D_{5,2,7} = (5,0,1;(1,5),(3,5))$.
\item[(iii)] $(D_{5,2,8},D_{6,1,2})$, where $D_{5,2,8} = (5,0,1;(2,5),(2,5))$.
\end{itemize}

\section{Spherical nestled actions}
\label{sec:spherical}
A spherical action is simply a nestled $(n,\ell)$-action whose quotient orbifold is a sphere. We will show in Proposition~\ref{prop:sphlb} that nestled $(n,\ell)$-actions must be spherical when $n$ is sufficiently large. This means that in order to derive bounds on $n$, it suffices to restrict attention to spherical actions. We will also derive several other results on spherical actions which we will be helpful in later sections.

\begin{definition}
A non-trivial nestled $(n, \ell)$-action is said to be \textit{spherical} if the underlying manifold of its quotient orbifold is topologically a sphere.
\end{definition}

\begin{example}
The actions in Examples ~\ref{ms} and ~\ref{ex:nestled actions} are spherical actions.
\end{example}

\begin{proposition}
\label{prop:sphlb}
If $n > \frac{2}{3}(2g-1)$, then every nestled $(n, \ell)$-action on $F$ is spherical.
\end{proposition}
\begin{proof}
Let $D = (n, \widetilde{g}, a; (c_1,x_1),\ldots,(c_n, x_{\ell}))$ be the data set associated with a nestled $(n, \ell)$-action on $F$. Equation~\ref{eq:nestled_action_riemann_hurwitz} gives
\begin{equation}
\label{eq:g_tilda_equation}
\widetilde{g} = \frac{1}{2} + \frac{2g-1}{2n} - \frac{\ell}{2} + \frac{1}{2}\sum_{i=1}^{\ell} \frac{1}{x_i} \ ,
\end{equation}
Each $x_i \geq 2$, and by Remark~\ref{rk:no_ell_equal_zero_nestled_actions}, we must have $\ell \geq1$, so this becomes
\[ \widetilde{g} \leq \frac{1}{2} + \frac{2g-1}{2n} - \frac{\ell}{4} \leq \frac{1}{4} + \frac{2g-1}{2n}\ .\]
That is, $\widetilde{g}\geq 1$ can hold only when $n\leq (4g-2)/3$.
\end{proof}

\begin{remark}
\label{rem:no-sph-elleq1}
There exists no spherical nestled $(n, \ell)$-action with $\ell = 1$. Suppose we assume on the contrary that $\ell=1$. Then, Equation ~\ref{chi:orb} would imply that
\[ \frac{1-2g}{n} = \frac{1}{x_1} \ .\] This is impossible since $x_1 > 0$ and $g \geq 1$.
\end{remark}

\begin{proposition}
\label{prop:small_spherical_actions}
Suppose that a surface $F$ of genus $g$ has a spherical nestled $(n,\ell)$-action. Write the prime factorization of $n$ as $n=p^a{q_1}^{a_1}\cdots {q_k}^{a_k}$ where $p^a > {q_i}^{a_i}$ for each $i\geq 1$, and write $q$ for $\min\{p,q_1,\ldots,q_k\}$. If \[n>\frac{2g-1}{2-\frac{2}{q}-\frac{1}{p^a}}\ ,\] then $\ell = 2$.
\end{proposition}
\begin{proof}
Each $x_i\geq q$, and by Proposition~\ref{prop1}, at least one $x_i \geq p^a$. Using Equation~\ref{eq:g_tilda_equation} we have
\begin{gather*}
0 = \frac{1}{2} + \frac{2g-1}{2n} - \frac{\ell}{2} + \frac{1}{2}\sum_{i=1}^{\ell} \frac{1}{x_i} \leq \frac{1}{2} + \frac{1}{2p^a} + \frac{2g-1}{2n}
- \frac{\ell}{2} + \frac{\ell-1}{2q}\\
\ell \leq 1 + \frac{q}{(q-1)p^a} + \frac{q}{q-1}\left(\frac{2g-1}{n}\right)
\end{gather*}
The right-hand side of the latter inequality is less than $3$ when the inequality in the proposition holds. Therefore, by Remark ~\ref{rem:no-sph-elleq1}, $\ell = 2$. \end{proof}

\begin{corollary}
\label{coro:small_spherical_actions}
Suppose that a surface $F$ of genus $g$ has a spherical nestled $(n,\ell)$-action.
\begin{enumerate}
\item[(i)] If $n=2$, then $\ell = 2g+1$. In particular, there does not
exist a spherical nestled $(2,2)$-action.
\item[(ii)] If $n=3$, then $\ell = g+1$. There exists a spherical nestled
$(3,2)$-action if and only if $g=1$.
\item[(iii)] If $n$ is even, $n\geq 4$, and $n>\frac{4}{3}(2g-1)$, then
$\ell=2$.
\item[(iv)] If $n$ is odd, $n\geq 5$, and $n> \frac{15}{17}(2g-1)$, then
$\ell=2$.
\end{enumerate}
\end{corollary}
\begin{proof}
For (i), an Euler characteristic calculation shows that $\ell = 2g+1$ when $n=2$. These are exactly the hyperelliptic actions.

For (ii), when $n=3$, an Euler characteristic calculation shows that $\ell = g+1$, and as seen in Section~\ref{sec:genus2_classification}, there is a
nestled $(3,2)$-action on the torus.

For (iii), suppose first that $n=6$. In Proposition~\ref{prop:small_spherical_actions} we have $q=2$ and $p^a=3$, giving the conclusion that if $6>\frac{3}{2}(2g-1)$,
then $\ell=2$. The condition $6>\frac{3}{2}(2g-1)$ holds exactly when $g\leq 2$, so (iii) is true in this case. One can check that there exist nestled $(6,2)$-actions
exactly when $g\leq 2$. For the cases of (iii) other than $n=6$, we have $q=2$ and $p^a\geq 4$, and Proposition~\ref{prop:small_spherical_actions} gives the result.

For (iv), we have $q\geq 3$ and $p^a\geq 5$. Again Proposition~\ref{prop:small_spherical_actions} gives the result.
\end{proof}

\section{Bounds on the degree of a root}
\label{sec:degree-bounds}
In this section, we use the Theorem~\ref{thm1} and the results derived in Sections~\ref{sec:nestlednl} and~\ref{sec:spherical} to derive some results on the degree $n$ of a root. Among the results derived is an upper bound and a stable upper bound for $n$.

\begin{remark}
\label{rem:maxorder}
It is a well known fact \cite{H1} that the maximum order for an automorphism of a surface of genus $g$ is $4g + 2$. In Example~\ref{ex:nestled actions}, we showed that a nestled action of order $4g+2$ always exists.
\end{remark}

\begin{proposition}
\label{prop3}
There exists no nestled $(4g + 1, \ell)$-action.
\end{proposition}

\begin{proof}
By Proposition~\ref{prop:sphlb}, a nestled $(4g + 1, \ell)$-action must be spherical, and by Proposition~\ref{prop:small_spherical_actions}, $\ell = 2$. Therefore,
Equation ~\ref{chi:orb} in the proof of Theorem~\ref{thm1} simplifies to give
\[ \frac{2g + 2}{4g + 1} = \frac{1}{x_1} + \frac{1}{x_2}\ .\]

Without loss of generality, we may assume that $x_1 \leq x_2$. Since $x_i \mid 4g+1$, $x_i \geq 3$. If $x_1 = 3$, then
\[ x_2 = \frac{3(4g + 1)}{2g + 5} = 3\left(2 - \frac{9}{2g +  5}\right) \ .\] Since $x_2 = 3$ is the only integer solution for $x_2$, Proposition ~\ref{prop1} would imply that $n=3$ which contradicts that fact that $n = 4g+1$. If $x_1 \geq 4$, then we would have that
\[\frac{1}{2} < \frac{2 + 2g}{4g + 1} = \frac{1}{x_1} + \frac{1}{x_2} \leq \frac{1}{2} \ ,\]
which is not possible.
\end{proof}

\begin{proposition}
Let $F = F_1 \#_C F_2$ be a closed oriented surface of genus $g \geq 2$. Let $(D_1,D_2)$ be a data set pair corresponding to a root of $t_C$ of degree~$n$, and let $n_i$ be the degree of $D_i$ for $i=1,2$. Then the $n_i$ cannot both satisfy $n_i\equiv 2\bmod 4$.
\label{prop:norootorder}
\end{proposition}

\begin{proof}
Suppose for contradiction that both $n_i$ satisfy $n_i\equiv 2\bmod 4$. Let $a_i$ denote the $a$-value of $D_i$, and let $k_i$ denote the inverse of $a_i$ modulo $n_i$.
Since $\gcd(k_i,n_i) = 1$, the $k_i$ must be odd. Also the fact that $\gcd(n_1, n_2) = 2k$ for some odd integer $k$ implies that $\frac{n}{n_i}$ is odd. From Equation
~\ref{dspair} for the data set pair $(D_1, D_2)$, we must have that\[ \frac{n}{n_1}k_1 + \frac{n}{n_2}k_2 \equiv 1 \bmod n\ ,\]  which is impossible since
$\frac{n}{n_1}k_1 + \frac{n}{n_2}k_2$ and $n$ are even.
\end{proof}

\begin{proposition}
Let $F = F_1 \#_C F_2$ be a closed oriented surface of genus $g \geq 2$. Suppose that $M(g_1,g_2)$ denotes the maximum degree of a root of the Dehn twist $t_C$ about $C$. Then $M(g_1,g_2) \leq 16g_1g_2 + 4(2g_1-g_2) - 2$.
\label{prop:root-bounds}
\end{proposition}

\begin{proof}
Let $n$ be the order of a root of $t_C$, given by a data set pair $(D_1,D_2)$.  We have $n= \lcm(n_1,n_2)$, where $n_i$ is the degree of $D_i$. By Remark~\ref{rem:maxorder}, each $n_i\leq 4g_i+2$. By Proposition~\ref{prop3}, neither $n_i=4g_i+1$, and by Proposition~\ref{prop:norootorder}, we cannot have both $n_1=4g_1+2$ and $n_2=4g_2+2$. If both $n_1=4g_1$ and $n_2=4g_2$, then $\lcm(n_1,n_2)=4\lcm(g_1,g_2)\leq 4g_1g_2 \leq 16g_1g_2 + 4(2g_1-g_2) - 2$. In general, since $g_1 \geq g_2$, we have that $M(g_1,g_2) \leq \max\{(4g_1+2)(4g_2-1), (4g_1-1)(4g_2+2)\}= 16g_1g_2 + 4(2g_1-g_2) - 2$.
\end{proof}

\begin{notation}
We will denote the upper bound $16g_1g_2 + 4(2g_1-g_2) - 2$ derived in Proposition~\ref{prop:root-bounds} by $U(g_1,g_2)$.
\end{notation}

\begin{theorem}
Let $F = F_1 \#_C F_2$ be a closed oriented surface of genus $g \geq 2$. Suppose that $n$ denotes the degree of a root of the Dehn twist $t_C$ about $C$. Then $n \leq 4g^2+2g$.
\label{thm:upper bound}
\end{theorem}

\begin{proof}
Since $g_2 = g - g_1$, we have that $16g_1g_2 + 4(2g_1-g_2) - 2 = -16g_1^2 + g_1(16g+12) - (4g+2)$, which has its maximum when $g_1 = \frac{1}{8}(4g+3)$. The fact that $g_1$ is an integer implies that when $g$ is even, $g_1 = g_2 = g/2$, and when $g$ is odd, $g_1 = (g+1)/2$ and $g_2 = (g-1)/2$. So Proposition~\ref{prop:root-bounds} tells us that when $g$ is even, $n \leq M(g/2, g/2) \leq 4g^2 + 2g - 2$, and when $g$ is even, $n \leq M((g+1)/2, (g-1)/2) \leq 4g^2 + 2g$.
\end{proof}

\begin{notation}
We will denote the upper bound $4g^2+2g$ derived in Theorem~\ref{thm:upper bound} by $U(g)$.
\end{notation}

For $2 \leq g \leq 35$, Table 1 gives the realizable maximum degrees of root, $m(g)$ (coming from compatible pairs of spherical nestled $(n,2)$-actions) and the upper bound $U(g)$. The last column gives the ratio $m(g)/U(g)$. These computations were made using software~\cite{GAP} written for the GAP programming language.
\begin{table}

\begin{center}
    \begin{tabular}{ | c | c | c | c |}
    \hline
   $g$ & $m(g)$ & $U(g)$ & $m(g)/U(g)$\\ \hline
    2 & 12 & 20 & 0.60 \\ \hline
    3 & 30 & 42 & 0.71 \\ \hline
    4 & 42 & 72 & 0.58 \\ \hline
    5 & 90 & 110 & 0.81 \\ \hline
    6 & 126 & 156 & 0.81 \\ \hline
    7 & 210 & 210 & 1.00 \\ \hline
    8 & 240 & 272 & 0.88 \\ \hline
    9 & 330 & 342 & 0.96 \\ \hline
    10 & 390 & 420 & 0.93 \\ \hline
    11 & 462 & 506 & 0.91 \\ \hline
    12 & 546 & 600 & 0.91 \\ \hline
    13 & 570 & 702 & 0.81 \\ \hline
    14 & 714 & 812 & 0.88 \\ \hline
    15 & 798 & 930 & 0.86 \\ \hline
    16 & 858 & 1056 & 0.81 \\ \hline
    17 & 966 & 1190 & 0.81 \\ \hline
    18 & 1122 & 1332 & 0.84 \\ \hline
    19 & 1254 & 1482 & 0.85 \\ \hline
    20 & 1326 & 1640 & 0.81 \\ \hline
    21 & 1518 & 1806 & 0.84 \\ \hline
    22 & 1650 & 1980 & 0.83 \\ \hline
    23 & 1794 & 2162 & 0.83 \\ \hline
    24 & 1950 & 2352 & 0.83 \\ \hline
    25 & 2046 & 2550 & 0.80 \\ \hline
    26 & 2262 & 2756 & 0.82 \\ \hline
    27 & 2418 & 2970 & 0.81 \\ \hline
    28 & 2550 & 3192 & 0.80 \\ \hline
    29 & 2730 & 3422 & 0.80 \\ \hline
    30 & 2958 & 3660 & 0.81 \\ \hline
    31 & 3162 & 3906 & 0.81 \\ \hline
    32 & 3306 & 4160 & 0.79 \\ \hline
    33 & 3570 & 4422 & 0.81 \\ \hline
    34 & 3774 & 4692 & 0.80 \\ \hline
    35 & 3990 & 4970 & 0.80 \\ \hline
\end{tabular}
\end{center}
\caption{The data seems to indicate that for large genera the ratio $m(g)/U(g)$ stabilizes to the 0.79-0.82 range.}
\end{table}

\begin{lemma}
\label{lem:odd-action bound}
Suppose that we have a spherical nestled $(4g - N, 2)$-action on a $F$ of genus $g$, where $N$ is a positive odd integer. Then $g \leq N + 3$.
\end{lemma}

\begin{proof}
Let $D = (4g-N, 0, a; (c_1, x_1), (c_2, x_2))$ be a data set for the nestled $(4g - N, 2)$-action on $F$. Since $4g-N$ is odd and $x_i \mid n$, we have that $x_i \geq 3$. If $x_1 \geq 3$, then Remark~\ref{prop1} implies that $x_2 \geq \frac{1}{3}(4g-N)$. So Equation~\ref{eq:nestled_action_riemann_hurwitz} gives the inequality
\[ \frac{2g - N + 1}{4g - N} \leq \frac{1}{3} + \frac{3}{4g-N}\ ,\] which upon simplification gives $g \leq N+3$.
\label{prop:stable-root-bounds}
\end{proof}

\begin{theorem}
Let $F = F_1 \#_C F_2$ be a closed oriented surface of genus $g \geq 2$. Suppose that $M(g_1,g_2)$ denotes the maximum order of a root of the Dehn twist $t_C$ about $C$. Then given a positive odd integer $N$, we have that $M(g_1,g_2) \leq  16g_1g_2 + 4(2g_1 - Ng_2)- 2N$ whenever both $g_i > N+3$.
\label{prop:stub}
\end{theorem}

\begin{proof}
By Remark~\ref{rem:maxorder}, each $n_i\leq 4g_i+2$. From Propositions~\ref{prop3} and~\ref{prop:norootorder}, we know that $n_i \neq 4g_i+1$ and that $n_i$ cannot both be $4g_i+2$. Suppose that the $n_i$ are not both even. If $\ell_i > 2$, then from Corollary~\ref{coro:small_spherical_actions} we have that $n_i \leq \frac{15}{17}(2g_i-1)$. If $\ell_i=2$, then Lemma~\ref{lem:odd-action bound} tells us that for all $g_i > N+3$, there exists no spherical nestled $(4g_i - N, 2)$-action on $F$. In particular, if $g_i > N+3$, then from Proposition ~\ref{prop:sphlb}, $n_i \leq \frac{2}{3}(2g_i-1) \leq \frac{15}{17}(2g_i-1)$. So for all $\ell$, if $g_i > N+3$, then $n_i \leq \frac{15}{17}(2g_i-1)$. We can see that $\frac{15}{17}(2g_i-1) \leq 4g_i-N$ whenever $g_i \geq \frac{1}{38}(17N-15)$. Therefore, if $g_i > \max \{N+3, \frac{1}{38}(17N-15) \} = N+3$, then we have that $M(g_1,g_2) \leq \max\{(4g_1-N)(4g_2+2), (4g_1+2)(4g_2-N)\} = 16g_1g_2 + 4\max\{(2g_1 - Ng_2), (2g_2 - Ng_1)\} - 2N = 16g_1g_2 + 4(2g_1 - Ng_2)- 2N$.

Suppose that both the $n_i$ are even. Then from Propositions~\ref{prop3} and~\ref{prop:norootorder}, we have that $M(g_1,g_2) \leq \lcm(4g_1+2,4g_2) \leq 8g_1g_2+4g_2$. We need to show that $8g_1g_2+4g_2 \leq 16g_1g_2 + 4(2g_1 - Ng_2)- 2N$. Since $g_1 > N+3$, $(16g_1g_2 + 4(2g_1 - Ng_2)- 2N) - (8g_1g_2+4g_2) = 8g_1g_2 +8g_1 - 4(N+1)g_2-2N > 8g_1g_2 + 8g_1 + 4(g_1-2)g_2 + 2(g_1-3) = 12g_1g_2 + 10g_1 - 8g_2 - 6 > 0$.
\end{proof}

\begin{notation}
We will denote the upper bound $16g_1g_2 + 4(2g_1 - Ng_2)- 2N$ derived in Theorem ~\ref{prop:stub} by $U(g_1,g_2,N)$.
\end{notation}

\begin{example}
When $N =11$, if both $g_i > 14$, then from Theorem~\ref{prop:stub}, $M(g_1,g_2) \leq U(g_1,g_2,11) = 16g_1g_2 + 4(2g_1 - 11g_2)- 22$. For genera pairs $(g_1,g_2)$ with $30 \leq g_1 + g_2 \leq 35$, Table 2 gives the values of the realizable maximum degree $m(g_1,g_2)$ (coming from compatible spherical nestled $(n,2)$-actions), the upper bound $U(g_1,g_2)$ (derived in Proposition~\ref{prop:root-bounds}), and the stable upper bound $U(g_1,g_2, N)$.
\begin{table}
\label{table:stable}
\begin{center}
    \begin{tabular}{| c | c | c | c | c |}
    \hline
    $g$ & $(g1,g2)$ & $m(g_1,g_2)$ & $U(g_1,g_2,11)$ & $U(g_1,g_2)$ \\ \hline
    30 & $(15, 15)$ & 2790 & 3038 & 3658 \\ \hline
    31 & $(16, 15)$ & 3162 & 3286 & 3906 \\ \hline
    32 & $(16, 16)$ & 3264 & 3498 & 4158 \\ \hline
    32 & $(17, 15)$ & 3162 & 3534 & 4154 \\ \hline
    33 & $(17, 16)$ & 3570 & 3762 & 4422 \\ \hline
    33 & $(18, 15)$ & 3534 & 3782 & 4402 \\ \hline
    34 & $(17, 17)$ & 3570 & 3990 & 4690 \\ \hline
    34 & $(18, 16)$ & 3774 & 4026 & 4686 \\ \hline
    34 & $(19, 15)$ & 3534 & 4030 & 4650 \\ \hline
    35 & $(18, 17)$ & 3990 & 4270 & 4970 \\ \hline
    35 & $(19, 16)$ & 3876 & 4290 & 4950 \\ \hline
    35 & $(20, 15)$ & 3690 & 4278 & 4898 \\ \hline
    \end{tabular}
\end{center}
\caption{For $N=11$, this data illustrates the stable bound $U(g_1,g_2,11)$ and the upper bound $U(g_1,g_2)$. When $g = 32$, we saw in Table 1 that the maximum realizable degree $m(g) = 3306$. This is larger than both the stable bounds $U(16,16,11)$ and $U(17,15,11)$.}
\end{table}
\end{example}
\section*{acknowledgements}
I would like to thank Steven Spallone for some useful discussions in elementary number theory.

\bibliographystyle{plain}
\bibliography{sepcurve}	

\end{document}